\documentclass[12pt]{article}
\usepackage[numbers,sort&compress]{natbib}
\usepackage{amsmath,amsthm,amsfonts,latexsym,amsopn,verbatim,amscd,amssymb,}
\usepackage{amssymb}
\usepackage{amsfonts}

\usepackage{graphics,color}
\usepackage{graphicx}

\theoremstyle{plain}
\newtheorem{Thm}{Theorem}[section]
\newtheorem{Lem}[Thm]{Lemma}
\newtheorem{Prop}[Thm]{Proposition}
\newtheorem{Cor}[Thm]{Corollary}
\theoremstyle{definition}

\newtheorem{Def}[Thm]{Definition}
\newtheorem{example}[Thm]{Example}

\numberwithin{equation}{section}

\newcommand{\bnum}{\begin{enumerate}}
\newcommand{\enum}{\end{enumerate}}
\begin{document}
\begin{center}
\textbf{Weakly Clean Ideal}\\
\end{center}
\begin{center}
Ajay Sharma and Dhiren Kumar Basnet\footnote{Corresponding author}\\
\small{\it Department of Mathematical Sciences, Tezpur University,
 \\ Napaam, Tezpur-784028, Assam, India.\\
Email: ajay123@tezu.ernet.in, dbasnet@tezu.ernet.in}
\end{center}

\noindent \textit{\small{\textbf{Abstract:}  }} Motivated by the concept of clean ideals, we introduce the notion of weakly clean ideals. We define an ideal $I$ of a ring $R$ to be weakly clean ideal if for any $x\in I$, $x=u+e$ or $x=u-e$, where $u$ is a unit in $R$ and $e$ is an idempotent in $R$. We discuss various properties of weakly clean ideals.

\bigskip

\smallskip

\noindent \small{\textbf{\textit{$2010$ Mathematics Subject Classification:}}  16N40, 16U99.} \\
\smallskip

\bigskip

\section{INTRODUCTION}
$\mbox{\hspace{.5cm}}$
Here rings $R$ are associative with unity unless otherwise indicated. The Jacobson radical, set of units, set of idempotents and centre of a ring $R$ are denoted by $J(R)$, $U(R)$, $Idem(R)$ and $C(R)$ respectively. Nicholson\cite{nicholson1977lifting} called an element $x$ of a ring $R$, a clean element, if $x=e+u$ for some $e\in Idem(R)$, $u\in U(R)$ and called the ring $R$ as clean ring if all its elements are clean. Weakening the condition of clean element, M.S. Ahn and D.D. Anderson\cite{ahn2006weakly} defined  an element $x$ as weakly clean if $x$ can be expressed as $x=u+e$ or $x=u-e$, where $u\in U(R)$, $e\in Idem(R)$. H. Chen and M. Chen\cite{chen2003clean}, introduced the concept of clean ideals as follows: an ideal $I$ of a ring $R$ is called clean ideal if for any $x\in I$, $x=u+e$, for some $u\in U(R)$ and $e\in Idem(R)$. Motivated by these ideas we define an ideal $I$ of a ring $R$ as weakly clean ideal if for any $x\in I$, $x=u+e$ or $x=u-e$, where $u\in U(R)$ and  $e\in Idem(R)$. Also an ideal $I$ of a ring $R$ is called uniquely weakly clean ideal if for each $a\in I$, there exists unique idempotent $e$ in $R$ such that $a-e\in U(R)$ or $a+e\in U(R)$. We discuss some interesting properties of weakly clean ideals. %We prove that for an ideal $I$ of a commutative ring $R$ and $n\geq 2$, then $\mathbb{M}_n(I)$ is weakly clean ideal of $\mathbb{M}_n(R)$ \textit{if and only if } $I$ is clean ideal of $R$.
 %Further we discuss about the relationship between weakly clean ideals of $R$ and weakly clean ideals of Morita context related to the ring $R$. Also studied the relationship between weakly clean ideals of $R$ and weakly clean ideals of the idealization of $R$.
\section{Weakly clean ideals}
\begin{Def}
  An ideal $I$ of a ring $R$ is called weakly clean ideal in case every element in $I$ is a sum or difference of a unit and an idempotent of $R$.
\end{Def}
Clearly every ideal of a weakly clean ring is weakly clean ideal. But there exists non weakly clean rings which contains some weakly clean ideals. Let $R_1$ be weakly clean ring and $R_2$ be non weakly clean ring. Then $R=R_1\oplus R_2$ is not a weakly clean ring. But clearly $I=R_1\oplus 0$ is weakly clean ideal of $R$.
\begin{Lem}
  If every proper ideal of a ring $R$ is clean(weakly clean) ideal then the ring $R$ is also clean(weakly clean) ring.
\end{Lem}
\begin{proof}
  Clearly every unit of a ring is clean. Let $x\in R\setminus U(R)$ then the ideal $<x>$ is proper ideal of $R$, so $x$ is clean in $R$.
\end{proof}
\begin{Cor}
  $R$ is clean(weakly clean) \textit{if and only if} every proper ideal of $R$ is clean(weakly clean).
\end{Cor}
The following is an example of weakly clean ideal which is not an clean ideal.
\begin{example}
  For the ring $R=\mathbb{Z}_{(3)}\cap \mathbb{Z}_{(5)}$, the ideal $<\frac{2}{11}>$ generated by $\frac{2}{11}$ is weakly clean ideal but not a clean ideal of $R$.
\end{example}
Following H. Chen and M. Chen\cite{chen2003clean}, we define weakly exchange ideal as follows:
\begin{Def}
  An ideal $I$ of a ring $R$ is called a weakly exchange ideal provided that for any $x\in I$, there exists an idempotent $e\in I$ such that $e-x\in R(x-x^2)$ or $e+x\in R(x+x^2)$.
\end{Def}
\begin{Lem}
  Every weakly clean ideal of a ring is a weakly exchange ideal.
\end{Lem}
\begin{proof}
  Let $I$ be a weakly clean ideal of $R$ and $x\in I$. Then $x=u+e$ or $x=u-e$, where $u\in U(R)$ and $e\in Idem(R)$. If $x=u+e$ then by Lemma 1.2 \cite{chen2003clean}, $x$ satisfies the exchange property. If $x=u-e$ then consider $f=u^{-1}(1-e)u\,\,$ so that $f^2=f$. Now $u(x+f)=x^2+x,\,\,$ so $x+f\in R(x^2+x)$.
\end{proof}
\begin{Thm}
  Let $R$ be a ring and $I$ an ideal in which every idempotent is central. Then the following are equivalent:
  \begin{enumerate}
    \item $I$ is weakly clean ideal.
    \item $I$ is weakly exchange ideal.
  \end{enumerate}
\end{Thm}
\begin{proof}
  $(1)\Rightarrow (2)$ is clear by Lemma 2.6.\\
  $(2)\Rightarrow (1)$ Given any $x\in I$, we have an idempotent $e\in Rx$ such that $1-e\in R(1-x)$ or $1-e\in R(1+x)$. If $1-e\in R(1-x)$ then by Theorem 1.3 \cite{chen2003clean}, $x$ is clean element. Suppose, $1-e \in R(1+x)$ then $e=ax$ and $1-e=b(1+x)$, for some $a,b\in R$. Assume that $ea=a$ and $(1-e)b=b$ so that $axa=ea=a\,$ and $b(1+x)b=b$. Here $ax,\,xa,\,b(1+x),\,(1+x)b\,$ all are central idempotents and $ax=(ax)(ax)=(ax)(xa)=x(ax)a=xa$, similarly $(1+x)b=b(1+x)$. Now $(a+b)(x+(1-e))=ax+bx+a(1-e)+b(1-e)=1\,$ so $x+(1-e)\,$ is a unit. Hence $x$ is a weakly clean element.
\end{proof}
\begin{Cor}
  Every weakly exchange ideal of a ring without nonzero nilpotent elements is a weakly clean ideal.
\end{Cor}
\begin{Lem}
  Let $R$ be a commutative ring and let $n\geq 1$. If $A\in \mathbb{M}_n(R)$ and $x\in R$, then $det(xE_{ij}+A)=xA_{ij}+det(A)$.
\end{Lem}
\begin{proof}
  See Lemma 7 \cite{kocsan2016weakly}.
\end{proof}
T. Ko{\c{s}}an, S. Sahinkaya and Y. Zhou\cite{kocsan2016weakly}, proved that  for a commutative ring $R$ and $n \geq 2$, $\mathbb{M}_n(R)$ is weakly clean if and only if $R$ is clean. Motivated by this result we generalise the similar result for weakly clean ideals of $\mathbb{M}_n(R)$ as follows:

\begin{Thm}
  Let $I$ be an ideal of a commutative ring $R$ and let $n\geq 2$. Then $\mathbb{M}_n(I)$ is weakly clean ideal of $\mathbb{M}_n(R)$ \textit{if and only if } $I$ is a clean ideal of $R$.
\end{Thm}
\begin{proof}
  Let $I$ be a clean ideal of $R$ then by Theorem 1.9 \cite{chen2003clean}, $\mathbb{M}_n(I)$ is clean ideal of $\mathbb{M}_n(R)$.\\
  Conversely, Let $\mathbb{M}_n(I)$ is weakly clean ideal of $\mathbb{M}_n(R)$. If possible, assume that $I$ is not clean ideal of $R$. Then there exists $x\in I$ such that $x\neq u+e$, for any $e\in Idem(R)$ and $u\in U(R)$. Consider $\mho=\{J\vartriangleleft R\,\,:\,\, \overline{x}\in R/J$ is not clean  $\}$. Notice that $\mho$ is non empty and $\mho$ is inductive set, so by Zorn's Lemma, $\mho$ contains a maximal member, say $I_1$. The maximality of $I_1$ implies that $R/I_1$ is an indecomposable ring. So $R/I_1$ is an indecomposable ring and $\overline{x}\in R/I_1$ is not clean. \par
  For contradicting the assumption we show that $A=xE_{11}-xE_{22}$ is not weakly clean in $\mathbb{M}_n(R)$. By Theorem 8 \cite{kocsan2016weakly}, it is clear that $A\in \mathbb{M}_n(R)$ is not weakly clean in $\mathbb{M}_n(R)$. Hence $I$ is clean ideal of $R$.
\end{proof}
%\begin{Cor}
%  Let $I$ be an ideal of a commutative ring $R$ and let $n\geq 2$. Then $\mathbb{T}_n(I)$ is weakly clean ideal of $\mathbb{T}_n(R)$ \textit{if and only if } $I$ is a clean ideal of $R$.
%\end{Cor}
%\begin{Thm}
%  Homomorphic image of clean(weakly clean) ideal is clean(weakly clean) ideal.
%\end{Thm}
\begin{Thm}
  Let $\{R_{\alpha}\}$ be a family of rings and $I_{\alpha}'s$ are ideals of $R_{\alpha}$, then the ideal $I=\prod I_{\alpha}$ of $R=\prod R_{\alpha}$ is weakly clean ideal \textit{if and only if} each $I_{\alpha}$ is weakly clean ideal of $\{R_{\alpha}\}$ and at most one $I_{\alpha}$ is not clean ideal.
\end{Thm}
\begin{proof}
  Let $I$ be weakly clean ideal of $R$. Then being homomorphic image of $I$ each $I_{\alpha}$ is weakly clean ideal of $R_{\alpha}$. Suppose $I_{\alpha_1}$ and $I_{\alpha_2}$ are not clean ideal, where ${\alpha}_1\neq {\alpha}_2$. Since $I_{\alpha_1}$ is not clean ideal, so not all elements $x\in I_{\alpha_1}$ is of the form $x=u-e$, where $u\in U(R_{\alpha_1})$ and $e\in Idem(R_{\alpha_1})$. As $I_{\alpha_1}$ is weakly clean ideal of $R_{\alpha_1}$, so there exists $x_{\alpha_1}\in I_{\alpha_1}$ with $x_{\alpha_1}=u_{\alpha_1}+e_{\alpha_1}$, where $u_{\alpha_1}\in U(R_{\alpha_1})$ and $e_{\alpha_1}\in Idem(R_{\alpha_1})$, but $x_{\alpha_1}\neq u-e$, for any $u\in U(R_{\alpha_1})$ and $e\in Idem(R_{\alpha_1})$. Similarly there exists $x_{\alpha_2}\in I_{\alpha_2}$ with $x_{\alpha_2}=u_{\alpha_2}-e_{\alpha_2}$, where $u_{\alpha_2}\in U(R_{\alpha_2})$ and $e_{\alpha_2}\in Idem(R_{\alpha_2})$, but $x_{\alpha_2}\neq u+e$, for any $u\in U(R_{\alpha_2})$ and $e\in Idem(R_{\alpha_2})$. Define $x=(x_\alpha)\in I$ by \begin{align*}
    x_\alpha &=x_{\alpha}\,\, \,\,\,\,\,\,\,if\,\, \alpha \in \{\alpha_1,\alpha_2\} \\
             &=0 \,\,\,\,\,\,\,\,\,\,\,\,\,if \,\,\alpha \notin \{\alpha_1,\alpha_2\}
  \end{align*}
  Then clearly $x\neq u\pm e$, for any $u\in U(R)$ and $e\in Idem(R)$. Hence at most one $I_\alpha$ is not clean ideal.\par
  $(\Leftarrow)$ If each $I_\alpha$ is clean ideal of $R_{\alpha}$ then $I=\prod I_\alpha$ is clean ideal of $R$ and hence weakly clean ideal of $R$. Assume $I_{\alpha_0}$ is weakly clean ideal but not clean ideal of $I_{\alpha_0}$ and that all other $I_{\alpha}$'s are clean ideals of $R_{\alpha}$. If $x=(x_\alpha) \in I$ then in $I_{\alpha_0}$, we can write $x_{\alpha_0}=u_{\alpha_0}+e_{\alpha_0}$ or $x_{\alpha_0}=u_{\alpha_0}-e_{\alpha_0}$, where $u_{\alpha_0}\in U(R_{\alpha_0})$ and $e_{\alpha_0}\in Idem(R_{\alpha_0})$. If $x_{\alpha_0}=u_{\alpha_0}+e_{\alpha_0}$, then for $\alpha \neq \alpha_0$ let, $x_{\alpha}=u_{\alpha}+e_{\alpha}$ and if $x_{\alpha_0}=u_{\alpha_0}-e_{\alpha_0}$, then for $\alpha \neq \alpha_0$ let, $x_{\alpha}=u_{\alpha}-e_{\alpha}$ then $u=(u_\alpha)\in U(R)$ and $e=(e_\alpha)\in Idem(R)$, such that $x=u+e$ or $x=u-e$ and consequently $I$ is weakly clean ideal of R.
\end{proof}
Next we define the concept of uniquely weakly clean ideal of a ring.
\begin{Def}
  An ideal $I$ of a ring $R$ is called uniquely weakly clean ideal if for each $a\in I$, there exists a unique idempotent $e$ in $R$ such that $a-e\in U(R)$ or $a+e\in U(R)$.
\end{Def}
\begin{Lem}
  Every idempotent in a uniquely weakly clean ideal is a central idempotent.
\end{Lem}
\begin{proof}
  Let $I$ be a uniquely weakly clean ideal of a ring $R$ and $e$ be any idempotent in $I$. For any $x\in R$, since $-e=-(e+ex(1-e))+ex(1-e)=(1-(e+ex(1-e)))-(1-ex(1-e))=(1-e)-1$, so $1-(e+ex(1-e))=1-e\Rightarrow ex=exe$. Similarly we can show that $xe=exe$. Hence $xe=ex$.
\end{proof}

A Morita context denoted by $(R,S,M,N,\psi,\phi)$ consists of two rings $R$ and $S$, two bimodules $_AN_B$ and $_BM_A$ and a pair of bimodule homomorphisms (called pairings) $\psi:N\otimes _SM\rightarrow R$ and $\phi:M\otimes _RN\rightarrow S$, which satisfies the following associativity: $\psi(n\otimes m)n'=n\phi(m\otimes n')$ and $\phi (m\otimes n)m'=m\psi(n\otimes m')$, for any $m,\,m'\in M$ and $n,\,n'\in N$. These conditions ensure that the set of matrices $\left(
                                                                 \begin{array}{cc}
                                                                   r & n \\
                                                                   m & s \\
                                                                 \end{array}
                                                               \right)$, where $r\in R$, $s\in S$, $m\in M$ and $n\in N$ forms a ring denoted by $T$, called the ring of the context.
H. Chen and M. Chen\cite{chen2003clean}, showed that for rings $R$ and $S$, if $T$ be the ring of Morita context $\left(
                                                                                                 \begin{array}{cc}
                                                                                                   R & M \\
                                                                                                   N & S \\
                                                                                                 \end{array}
                                                                                               \right)$ with zero pairing
  and $I$ and $J$ are clean ideals of  rings $R$ and $S$ respectively, then                                                         $\left(
                                                                                                                                   \begin{array}{cc}
                                                                                                                                     I & M \\
                                                                                                                                     N & J \\
                                                                                                                                   \end{array}
                                                                                                                                 \right)$
is a clean ideal of $T$. Here we prove the similar result for weakly clean ideal.
\begin{Thm}
  Let $T=\left(
           \begin{array}{cc}
             R & M \\
             N & S \\
           \end{array}
         \right)
  $ be a Morita context. If $I$ and $J$ be weakly clean ideals of $R$ and $S$ respectively and either $I$ or $J$ is clean ideal, then the ideal $\left(
                                                                                                                                    \begin{array}{cc}
                                                                                                                                      I & M \\
                                                                                                                                      N & J \\
                                                                                                                                    \end{array}
                                                                                                                                  \right)$
   is weakly clean ideal of $T$.
\end{Thm}
\begin{proof}
  Without loss of generality, we can assume that $J$ is clean ideal of $S$. To show $\left(
                                                                                                                                   \begin{array}{cc}
                                                                                                                                     I & M \\
                                                                                                                                     N & J \\
                                                                                                                                   \end{array}
                                                                                                                                 \right)$
  is weakly clean ideal of $T$. Let $A=\left(
                                         \begin{array}{cc}
                                           a & m \\
                                           n & b \\
                                         \end{array}
                                       \right)
   \in \left(
          \begin{array}{cc}
            I & M \\
            N & J \\
          \end{array}
        \right)$, where $a\in I$, $b\in J$, $m\in M$ and $n\in N$. As $I$ is weakly clean ideal of $R$, so $a=e+u$ or $a=-e+u$, where $e\in Idem(R)$ and $u\in U(R)$.  \par $Case\,\,I$:  If $a=e+u$, then set $b=f+v$, where $f\in Idem(S)$ and $v\in U(S)$. Let, $E=\left(
                                                                                        \begin{array}{cc}
                                                                                          e & o \\
                                                                                          0 & f \\
                                                                                        \end{array}
                                                                                      \right)$
     and $U=\left(
              \begin{array}{cc}
                u & m \\
                n & v \\
              \end{array}
            \right)
     $. It is easy to verify that $E=E^2\in T$ and \begin{center}
                                                     $U\left(
                                                        \begin{array}{cc}
                                                          u^{-1} & -u^{-1}mv^{-1} \\
                                                          -v^{-1}nu^{-1} & v^{-1} \\
                                                        \end{array}
                                                      \right)=\left(
                                                        \begin{array}{cc}
                                                          u^{-1} & -u^{-1}mv^{-1} \\
                                                          -v^{-1}nu^{-1} & v^{-1} \\
                                                        \end{array}
                                                      \right)U=\left(
                                                                 \begin{array}{cc}
                                                                   1 & 0 \\
                                                                   0 & 1 \\
                                                                 \end{array}
                                                               \right)$.
                                                   \end{center}So $U\in  U(T).$\par

  $Case\,\, II$:  If $a=-e+u$, then we set $b=-f+v$, where $f\in Idem(S)$ and $v\in U(S)$. Let, $E=-\left(
                                                                                              \begin{array}{cc}
                                                                                                e & 0 \\
                                                                                                0 & f \\
                                                                                              \end{array}
                                                                                            \right)$   and $U=\left(
                                                                                                               \begin{array}{cc}
                                                                                                                 u & m \\
                                                                                                                 n & v \\
                                                                                                               \end{array}
                                                                                                             \right)$. Similar as above $E^2=E\in Idem(T)$
  and $U\in U(T)$.
\end{proof}
Let $A_1$, $A_2$ and $A_3$ be associative rings with identities and $A_{21}$, $A_{31}$ and $A_{32}$ be $(A_2,A_1)$-, $(A_3, A_1)$- and $(A_3, A_2)$-bimodules respectively. Let $\phi : A_{32}\otimes_{A_2} A_{21} \rightarrow A_{31} $ be an $(A_3, A_1)$-homomorphism then
$T=\left(
                            \begin{array}{ccc}
                              A_1 & 0 & 0 \\
                              A_{21} & A_2 & 0 \\
                              A_{31} & A_{32} & A_3\\
                            \end{array}
                          \right)$ is a lower triangular matrix ring with usual matrix operations.
\begin{Thm}
  If $I$, $J$ and $K$ are weakly clean ideals of rings $A_1$, $A_2$ and $A_3$ respectively, where at least two of them are clean ideals then the formal triangular matrix ideal $\left(
                            \begin{array}{ccc}
                              I & 0 & 0 \\
                              A_{21} & J & 0 \\
                              A_{31} & A_{32} & K \\
                            \end{array}
                          \right)$ is a weakly clean ideal of $\left(
                            \begin{array}{ccc}
                              A_1 & 0 & 0 \\
                              A_{21} & A_2 & 0 \\
                              A_{31} & A_{32} & A_3\\
                            \end{array}
                          \right)$.
\end{Thm}
\begin{proof}
  Assume that $I$ and $K$ are clean ideals $A_1$ and $A_3$ and $J$ is weakly clean ideal of $A_2$. Let, $B=\left(
                                                                              \begin{array}{cc}
                                                                                A_2 & 0 \\
                                                                                A_{32} & A_3 \\
                                                                              \end{array}
                                                                            \right)$ and $M=\left(
                                                                                              \begin{array}{c}
                                                                                                A_{21} \\
                                                                                                A_{31} \\
                                                                                              \end{array}
                                                                                            \right)$.
  As $J$ is weakly clean ideal of  $A_2$ and $K$ is clean ideal of $A_3$, so by Theorem 2.14, we see that $P=\left(
                                                                                              \begin{array}{cc}
                                                                                                 J & 0 \\
                                                                                A_{32} & K \\
                                                                                              \end{array}
                                                                                            \right)$
  is a weakly clean ideal of $B$. Again by Theorem 2.14, $\left(
                                                           \begin{array}{cc}
                                                             I & 0 \\
                                                             M & P \\
                                                           \end{array}
                                                         \right)$ is a weakly clean ideal of $\left(
                                                                                                \begin{array}{cc}
                                                                                                  A_1 & 0 \\
                                                                                                  M & B \\
                                                                                                \end{array}
                                                                                              \right)
                                                         $,
  that is $\left(
                            \begin{array}{ccc}
                              I & 0 & 0 \\
                              A_{21} & J & 0 \\
                              A_{31} & A_{32} & K \\
                            \end{array}
                          \right)$ is a weakly clean ideal of $\left(
                            \begin{array}{ccc}
                              A_1 & 0 & 0 \\
                              A_{21} & A_2 & 0 \\
                              A_{31} & A_{32} & A_3\\
                            \end{array}
                          \right)$.
\end{proof}
\begin{Thm}
  Let $A_1$, $A_2$ and $A_3$ are rings. If the formal triangular matrix ideal $\left(
                                                                                  \begin{array}{ccc}
                                                                                    I & 0 & 0 \\
                                                                                    A_{21} & J & 0 \\
                                                                                    A_{31} & A_{32} & K \\
                                                                                  \end{array}
                                                                                \right) $
  is a weakly clean ideal of $T=\left(
                                                                                  \begin{array}{ccc}
                                                                                    A_1 & 0 & 0 \\
                                                                                    A_{21} & A_2 & 0 \\
                                                                                    A_{31} & A_{32} & A_3 \\
                                                                                  \end{array}
                                                                                \right)$
  then $I$, $J$ and $K$ are weakly clean ideals of $A_1$, $A_2$ and $A_3$ respectively.
\end{Thm}
\begin{proof}
  For $x\in I$, we have $\left(
                           \begin{array}{ccc}
                             x & 0 & 0 \\
                             0 & 0 & 0 \\
                             0 & 0 & 0 \\
                           \end{array}
                         \right)\in \left(
                                                                                  \begin{array}{ccc}
                                                                                    I & 0 & 0 \\
                                                                                    A_{21} & J & 0 \\
                                                                                    A_{31} & A_{32} & K \\
                                                                                  \end{array}
                                                                                \right)$
Thus, \begin{center}
        $\left(
                           \begin{array}{ccc}
                             x & 0 & 0 \\
                             0 & 0 & 0 \\
                             0 & 0 & 0 \\
                           \end{array}
                         \right)=\left(
                           \begin{array}{ccc}
                             e_1 & 0 & 0 \\
                             \star & e_2 & 0 \\
                             \star & \star & e_3 \\
                           \end{array}
                         \right)+\left(
                           \begin{array}{ccc}
                             u_1 & 0 & 0 \\
                             \star & u_2 & 0 \\
                             \star & \star & u_3 \\
                           \end{array}
                         \right)$

      \end{center}
      or
      \begin{center}
        $\left(
                           \begin{array}{ccc}
                             x & 0 & 0 \\
                             0 & 0 & 0 \\
                             0 & 0 & 0 \\
                           \end{array}
                         \right)=-\left(
                           \begin{array}{ccc}
                             e_1 & 0 & 0 \\
                             \star & e_2 & 0 \\
                             \star & \star & e_3 \\
                           \end{array}
                         \right)+\left(
                           \begin{array}{ccc}
                             u_1 & 0 & 0 \\
                             \star & u_2 & 0 \\
                             \star & \star & u_3 \\
                           \end{array}
                         \right)$

      \end{center}
 where $\left(
                           \begin{array}{ccc}
                             e_1 & 0 & 0 \\
                             \star & e_2 & 0 \\
                             \star & \star & e_3 \\
                           \end{array}
                         \right)\in Idem(T)$ and $\left(
                           \begin{array}{ccc}
                             u_1 & 0 & 0 \\
                             \star & u_2 & 0 \\
                             \star & \star & u_3 \\
                           \end{array}
                         \right)\in U(T)$.
  It is clear that $e_1^2=e_1\in Idem(A_1)$ and $u_1\in U(A_1)$. Also $x=e_1+u_1$ or $x=-e_1+u_1$, so $I$ is weakly clean ideal of $A_1$. Similarly we can show that $J$ and $K$ are weakly clean ideals of $A_2$ and $A_3$ respectively.
\end{proof}
A finite orthogonal set of idempotents $e_1, \cdot\cdot\cdot , e_n$ in a ring $R$ is said to be complete set if $e_1+ \cdot\cdot\cdot +e_n=1$.
\begin{Prop}
  Let $R$ be a ring and $I$ an ideal of $R$. Then the following are equivalent:
  \begin{enumerate}
    \item $I$ is a weakly clean ideal of $R$.
    \item There exists a complete set $\{e_1, e_2, \cdot\cdot\cdot, e_n\}$ of idempotents such that $e_iIe_i$ is a weakly clean ideal of $e_iRe_i$, for all $i$ and at most one $e_iIe_i$ is not clean ideal of $e_iRe_i$.
  \end{enumerate}
\end{Prop}
\begin{proof}
  $(1)\Rightarrow (2)$ is trivial by taking $n=1$ and $e_1=1$.\par
  $(2)\Rightarrow (1)$  It is enough to show the result for $n=2$. Without loss of generality assume that $e_1Ie_1$ is weakly clean ideal of $e_1Re_1$ and $e_2Ie_2$ is clean ideal of $e_2Re_2$. It is clear that $I\cong \left(
                                                        \begin{array}{cc}
                                                          e_1Ie_1 & e_1Ie_2 \\
                                                          e_2Ie_1 & e_2Ie_2 \\
                                                        \end{array}
                                                      \right)$
  and $R\cong \left(
                                                        \begin{array}{cc}
                                                          e_1Re_1 & e_1Re_2 \\
                                                          e_2Re_1 & e_2Re_2 \\
                                                        \end{array}
                                                      \right)$ as $\{e_1, e_2\}$ be a complete set.
  Let $A=\left(
           \begin{array}{cc}
             a_{11} & a_{12} \\
             a_{21} & a_{22} \\
           \end{array}
         \right)\in I$.
As $e_1Ie_1$ is weakly clean ideal, so $a_{11}=u+e$ or $a_{11}=u-e$, where $e\in Idem(e_1Re_1)$ and $u\in U(e_1Re_1)$. Also $a_{22}-a_{21}u^{-1}a_{12}\in e_2Ie_2$. \\
Case I: If $a_{11}=e+u$, then we can set $a_{22}-a_{21}u^{-1}a_{12}=f+v$, where $f\in Idem(e_2Re_2)$ and $v\in U(e_2Re_2)$ then by Proposition $1.15$ \cite{chen2003clean}, $A$ is a clean element of $\left(
                                                        \begin{array}{cc}
                                                          e_1Re_1 & e_1Re_2 \\
                                                          e_2Re_1 & e_2Re_2 \\
                                                        \end{array}
                                                      \right)$. \\
Case II: If $a_{11}=-e+u$ then we can set $a_{22}-a_{21}u^{-1}a_{12}=-f+v$, where $f\in Idem(e_2Re_2)$ and $v\in U(e_2Re_2)$. Set $E=\left(
                                                                                       \begin{array}{cc}
                                                                                         e & 0 \\
                                                                                         0 & f \\
                                                                                       \end{array}
                                                                                     \right)$ and $U=\left(
                                                                                                       \begin{array}{cc}
                                                                                                         u & a_{12} \\
                                                                                                         a_{21} & v+a_{21}u^{-1}a_{12} \\
                                                                                                       \end{array}
                                                                                                     \right)$
By Proposition $1.15$ \cite{chen2003clean}, $E^2=E$ and $U$ is a unit in $\left(
                                                        \begin{array}{cc}
                                                          e_1Re_1 & e_1Re_2 \\
                                                          e_2Re_1 & e_2Re_2 \\
                                                        \end{array}
                                                      \right)$. Also $A=-E+U$, as required.

\end{proof}
\begin{Prop}
  Let $I$ be an ideal of a commutative ring $R$. Then $I$ is weakly clean ideal of $R$ \textit{if and only if} the ideal $I[[x]]$ is weakly clean ideal of $R[[x]]$.
\end{Prop}
\begin{proof}
  Let $I$ be a weakly clean ideal of $R$. Let $f(x)=\sum a_ix^i\in I[[x]]$, clearly $a_0\in I$, so $a_0=u_0+e_0$ or $a=u_0-e_0$, where $e_0\in Idem(R)$ and $u_0\in U(R)$. If $a_0=u_0+e_0$, then $f(x)=\sum a_ix^i=e_0+u_0+a_1x+a_2x^2+\cdot\cdot\cdot$, where $u_0+a_1x+a_2x^2+\cdot\cdot\cdot\in U(R[[x]])$ and $e_0\in Idem(R)\subseteq Idem(R[[x]])$. Similarly for $a=u_0-e_0$, $f(x)$ is weakly clean element in $R[[x]]$. Conversely if $I[[x]]$ is a weakly clean ideal of $R[[x]]$ then as a homomorphic copy of $I[[x]]$, $I$ is also a weakly clean ideal of $R$.
\end{proof}
Let $R$ be a commutative ring and $M$ be a $R$-module. Then the idealization of $R$ and $M$ is the ring $R(M)$ with underlying set $R\times M$ under coordinatewise addition and multiplication given by $(r,m)(r',m')=(rr', rm'+r'm)$, for all $r, r'\in R$ and $m, m' \in M$. It is obvious that if $I$ is an ideal of $R$ then for any submodule $N$ of $M$, $I(N)=\{(r,n)\, : \,r\in I \,\,and \, \,n\in N \}$ is an ideal of $R(M)$. We mention basic existing result about idempotent and unit element in $R(M)$ and study the weakly clean ideals of the idealization $R(M)$ of $R$ and $R$-module $M$.
\begin{Lem}
  Let $R$ be a commutative ring and $R(M)$ be the idealization of $R$ and $R$-module $M$. Then the following hold:
  \begin{enumerate}
    \item  $(r,m) \in Idem(R(M))$ if and only if $r \in Idem(R)$ and $m =0$.
    \item  $(r,m) \in U(R(M))$ if and only if $r \in U(R)$.
  \end{enumerate}
\end{Lem}
\begin{Prop}
Let $R$ be a commutative ring and $R(M)$ is a idealization of $R$ and $R$-module $M$. Then an ideal $I$ of $R$ is weakly clean ideal(clean ideal) of $R$ \textit{if and only if} $I(N)$ is weakly clean ideal(clean ideal) of $R(M)$, for any submodule $N$ of $M$.
\end{Prop}
\begin{proof}
  $(\Rightarrow)$ Consider $(x,n)\in I(N)$. For $x\in I$, $x=u+e$ or $x=u-e$, where $u\in U(R)$ and $e\in Idem(R)$, so $(x,n)=(e,0)+(u,n)$ or $(x,n)=-(e,0)+(u,n)$, where $(e,0)\in Idem(R(M))$ and $(u,n)\in U(R(M))$, by Lemma 2.20.\\

  $(\Leftarrow)$ Let $r\in I$, for $(r,n)\in I(N)$, $(r,n)=(e,0)+(u,n')$ or $(r,n)=-(e,0)+(u,n')$, where $(e,0)\in Idem(R(M))$, $(u,n')\in U(R(M))$ and $n,n'\in M$. Hence $r=e+u$ or $r=-e+u$, where $e\in Idem(R)$ and $u\in U(R)$ by Lemma 2.20, as required.
\end{proof}
\begin{Thm}
  Let $I$ be an ideal of a ring $R$ containing $J(R)$ and idempotents can be lifted modulo $J(R)$. Then $I$ is weakly clean ideal of $R$ \textit{if and only if} $I/J(R)$ is weakly clean ideal of $R/J(R)$.
\end{Thm}
\begin{proof}
  $(\Leftarrow)$ Let, $x\in I$, so $\overline{x}=\overline{e}+\overline{u}$ or $\overline{x}=-\overline{e}+\overline{u}$, where $\overline{e}\in Idem(R/J(R))$ and $\overline{u}\in U(R/J(R))$. Hence, $x-e-u\in J(R)$ or $x+e-u\in J(R)$, so $x=e+u+r$ or $x=-e+u+r$, where $r\in J(R)$. Since idempotents can be lifted modulo $J(R)$, we may assume that $e$ is an idempotent of $R$. So $I$ is weakly clean ideal of $R$.\par
  Converse is clear because if $u\in U(R)$ then $u+J(R)\in U(R/J(R))$ and $e+J(R)\in Idem(R/J(R))$, for $e\in Idem(R)$.
\end{proof}
If $I+J$, sum of two ideals $I$ and $J$, is weakly clean ideal of $R$ then $I$ and $J$ are also weakly clean ideal of $R$, as $I,J\subseteq I+J$. The converse is not true as shown by the example given below.
\begin{example}
  For $R=\mathbb{Z}_{(3)}\cap \mathbb{Z}_{(5)}$, the ring $R\times R$ is not weakly clean ring by Theorem $1.7$ \cite{ahn2006weakly}.  Clearly the ideals $<\frac{2}{11}>$ and $<\frac{4}{7}>$ generated by $\frac{2}{11}$ and $\frac{4}{7}$ respectively are weakly clean ideals but not clean ideals of $R$. Let $I_1=<\frac{2}{11}>\times \{0\}$ and $I_2=\{0\}\times <\frac{4}{7}>$, then $I_1$ and $I_2$ are weakly clean ideals of $R\times R$ but not clean ideals of $R\times R$. Hence $I_1+I_2=<\frac{2}{11}> \times<\frac{4}{7}>$ is not weakly clean ideal of $R\times R$ by Theorem $2.11$.
\end{example}
However we have a partial converse as follows.
\begin{Prop}
  If $I$ and $J$ are two weakly clean ideals of a ring $R$ and any one of $I$ and $J$ is contained in $J(R)$ then $I+J$ is also weakly clean ideal of $R$.
\end{Prop}
\begin{proof}
  Without loss of generality assume $J\subseteq J(R)$ and $x\in I+J$. Then $x=a+b$, where $a\in I$ and $b\in J\subseteq J(R)$. So, there exist $e\in Idem(R)$ and $u\in U(R)$ such that $a=u+e$ or $a=u-e$. Hence $x=e+u+b$ or $x=-e+u+b$, which gives $x$ is a weakly clean element of $R$.
\end{proof}
%\bibliographystyle{TUabbrvnat}
%\bibliography{ajbib}

\end{document}